\documentclass{article}

\ifdefined\AMM
\usepackage{maa-monthly}
\else
\usepackage{maa-monthly-arxiv}
\fi

\usepackage{picinpar,moresize,xfrac,graphpap,dcolumn,wrapfig,graphicx}

\DeclareGraphicsExtensions{.png,.pdf,.jpg}
\usepackage{subcaption}
\usepackage{stackengine}

\usepackage{tikz,pgfplots}
\usepackage{tikz-cd}
\usepackage{pgf}
\usepgfplotslibrary{colormaps}
\usetikzlibrary{pgfplots.colormaps}
\usetikzlibrary{arrows,automata,positioning,backgrounds}
\usepackage{rotating}
\pgfplotsset{compat=newest}
\pgfplotsset{plot coordinates/math parser=false}
\newlength\figureheight
\newlength\figurewidth

\usepackage[draft]{pdfcomment}
\usepackage{hyperref}
\hypersetup{
    colorlinks=true,
    linkcolor=black,
    filecolor=magenta,      
    urlcolor=magenta,
}

\usepackage[subrefformat=parens]{subcaption}
\definecolor{revcolor}{rgb}{0.0, 0.3, 0.8}

\usepackage{soul,array,calc,url,ragged2e,graphpap}
\usepackage{booktabs} 
\usepackage{tabularx}
\usepackage{colortbl}
\usepackage{multirow}
\usepackage{hyphenat}
\usepackage{transparent}
\usepackage{enumerate}
\usepackage{mathrsfs}
\usepackage{mdframed}
\usepackage{acro}

\usepackage[figure]{hypcap}
\usepackage{mathtools}
\mathtoolsset{centercolon}
\usepackage{nicefrac}
\usepackage{units}
\usepackage[normalem]{ulem}
\usepackage{cancel}
\usepackage{pbox}
\usepackage{multicol}
\usepackage{cases}
\usepackage[all]{xy}
\usepackage{nicematrix}

\usepackage{adjustbox}
\usepackage{wrapfig}
\AfterEndEnvironment{wrapfigure}{\setlength{\intextsep}{0mm}}

\usepackage{cleveref}
\crefmultiformat{subequation}%
{\edef\crefstripprefixinfo{#1}(#2#1#3}%
{,#2\crefstripprefix{\crefstripprefixinfo}{#1}#3)}%
{,#2\crefstripprefix{\crefstripprefixinfo}{#1}#3}%
{,#2\crefstripprefix{\crefstripprefixinfo}{#1}#3)}

\makeatletter
\DeclareFontFamily{OMX}{MnSymbolE}{}
\DeclareSymbolFont{MnLargeSymbols}{OMX}{MnSymbolE}{m}{n}
\SetSymbolFont{MnLargeSymbols}{bold}{OMX}{MnSymbolE}{b}{n}
\DeclareFontShape{OMX}{MnSymbolE}{m}{n}{
    <-6>  MnSymbolE5
   <6-7>  MnSymbolE6
   <7-8>  MnSymbolE7
   <8-9>  MnSymbolE8
   <9-10> MnSymbolE9
  <10-12> MnSymbolE10
  <12->   MnSymbolE12
}{}
\DeclareFontShape{OMX}{MnSymbolE}{b}{n}{
    <-6>  MnSymbolE-Bold5
   <6-7>  MnSymbolE-Bold6
   <7-8>  MnSymbolE-Bold7
   <8-9>  MnSymbolE-Bold8
   <9-10> MnSymbolE-Bold9
  <10-12> MnSymbolE-Bold10
  <12->   MnSymbolE-Bold12
}{}
\let\llangle\@undefined
\let\rrangle\@undefined
\DeclareMathDelimiter{\llangle}{\mathopen}%
                     {MnLargeSymbols}{'164}{MnLargeSymbols}{'164}
\DeclareMathDelimiter{\rrangle}{\mathclose}%
                     {MnLargeSymbols}{'171}{MnLargeSymbols}{'171}
\makeatother

\usepackage{fancyvrb,listings}
\usepackage{algorithm}
\usepackage{algorithmicx}
\usepackage[noend]{algpseudocode}

\algrenewcommand\alglinenumber[1]{\footnotesize #1:}
\makeatletter  
 \renewcommand{\ALG@name}{\small Algorithm} 
\makeatother 
\algdef{SE}[DOWHILE]{Do}{doWhile}{\algorithmicdo}[1]{\algorithmicwhile\ #1}%

\theoremstyle{theorem}
\newtheorem{theorem}{Theorem}
\newtheorem{lemma}{Lemma}

\newtheorem{corollary}{Corollary}
\newtheorem{problem}{Problem}
\newtheorem{example}{Example}
\theoremstyle{definition}
\newtheorem{definition}{Definition}

\newenvironment{claim}
{\par\medskip
 \phantomsection
 \noindent\textit{Claim.}\itshape
}
{\par\medskip}

\newtheorem{remark}{Remark}

\newcommand{\figref}[1]{\textup{Fig.~\ref{#1}}}

\def\ie{\emph{i.e.}}
\def\eg{\emph{e.g.}}

\def\CC{\mathbb{C}}

\def\NN{\mathbb{N}}

\def\RR{\mathbb{R}}

\def\bA{\mathbf{A}}

\def\bH{\mathbf{H}}

\def\bP{\mathbf{P}}
\def\bQ{\mathbf{Q}}

\def\ba{\mathbf{a}}

\def\be{\mathbf{e}}
\def\bff{\mathbf{f}}

\def\bx{\mathbf{x}}

\def\bz{\mathbf{z}}

\def\bzero{\mathbf{0}}

\usepackage{bbm}
\DeclareSymbolFont{bbold}{U}{bbold}{m}{n}
\DeclareSymbolFontAlphabet{\mathbbold}{bbold}
\newcommand{\ii}{\mkern1.5mu\mathbbold{i}\mkern1.5mu}

\renewcommand{\Im}{\operatorname{Im}}
\renewcommand{\Re}{\operatorname{Re}}

\DeclarePairedDelimiterX\braket[2]{\langle}{\rangle}{#1\,\delimsize\vert\,\mathopen{}#2}

\definecolor{b1}{rgb}{0.158099,0.313781,0.636957}
\definecolor{b2}{rgb}{0.525367,0.691857,0.998936}
\definecolor{g1}{rgb}{0.256000,0.640000,0.576000}
\definecolor{g2}{rgb}{0.559573,0.800781,0.760580}
\definecolor{p1}{rgb}{0.416000,0.192000,0.640000}
\definecolor{p2}{rgb}{0.680880,0.561880,0.799881}
\definecolor{r1}{rgb}{0.800000,0.200000,0.000000}
\definecolor{r2}{rgb}{1.000000,0.702595,0.603461}
\definecolor{k1}{gray}{0}
\definecolor{k2}{gray}{0.7}

\usepackage{enumitem}

\begin{document}

\ifdefined\AMM
\begin{center}
{\LARGE{Title Page for \textit{The American Mathematical Monthly}}}
\end{center}

\vspace{150px}
\begin{center}

\textbf{l'H\^opital rules for complex-valued functions in higher dimensions} \\
\vspace{25px}

Albert Chern,\\ 
UCSD,
California, USA, 
alchern@ucsd.edu

\vspace{15px}

Sadashige Ishida (corresponding author), \\
ISTA,
Klosterneuburg, Austria, 
sadashige.ishida@ist.ac.at

\vspace{15px}

\end{center}
\pagebreak
\fi

\title{l'H\^opital rules for complex-valued functions in higher dimensions} 
\markright{Complex L’H\^opital Rules in Higher Dimensions}
\author{Albert Chern and Sadashige Ishida} 
\maketitle

\begin{abstract}
In calculus, l'H\^opital's rule provides a simple way to evaluate the limits of quotient functions when both the numerator and denominator vanish.  But what happens when we move beyond real functions on a real interval?  In this article, we study when the quotient of two complex-valued functions in higher dimension can be defined continuously at the points where both functions vanish.  Surprisingly, the answer is far subtler than in the real-valued setting.  We provide a complete characterization for the continuity of the quotient function.  We also point out why extending this result to smoother quotients remains an intriguing challenge.

\end{abstract}

\vspace{10px}


\vspace{10px}

\section{Introduction}
The classical Bernoulli--l'H\^opital rule for real functions \(f(x)\) and \(g(x)\) on a real interval states conditions under which the ratio \(f(x)/g(x)\) is well-defined at a common zero of \(f\) and \(g\).
Specifically, if \(f(a) = g(a) = 0\), \(g'(a)\neq 0\), and the limit \(L = \lim_{x\to a}f'(x)/g'(x)\) exists, then \(\lim_{x\to a}f(x)/g(x) = L\).

The result was further extended to complex-valued functions on a real interval by \cite{Carter:1958:HRC} and, more recently, to real-valued multivariable functions by \cite{Lawlor:2020:HRM}. 

While the classical l'H\^opital rule and its extensions focus on the limit of \(f(x)/g(x)\) at the common zeros of \(f\) and \(g\), a stronger theorem can be stated in terms of regularity.  If we define \(f(a)/g(a)\coloneqq L\), then \(f(x)/g(x)\) remains a smooth function as long as both \(f\) and \(g\) are smooth. 
\begin{theorem}[Smooth quotient of real functions]
\label{thm:RealQuotient1D}
    Let \(f,g\in C^\infty(I)\) be smooth real-valued functions on a real interval \(I\).  Suppose \(f,g\) have only simple zeros in \(I\) and that their zero sets are identical.  Then there exists a non-vanishing smooth function \(\varphi\in C^\infty(I)\) such that \(f = g\varphi\) and $\varphi=f'/g'$ at the zero level sets.
\end{theorem}
Here, a real function \(f\) on \(I\) is said to have simple zeros if \(f'(a)\neq 0 \) whenever \(f(a) = 0\).
An elementary proof of \autoref{thm:RealQuotient1D} is provided in \autoref{app:FirstOrderTaylorExpansion} and \ref{app:ProofOfRealQuotient1D}.

Intuitively, functions with common zeros can be thought of as sharing common factors, which cancel out in their quotient.  A similar result holds for real-valued multivariable functions.  Let \(\Omega\subseteq \RR^n\) be an open set, with \(n\geq 2\), serving as the domain.  The zero set of a real-valued function \(f\colon\Omega\to\RR\) is denoted by \(\Sigma_f = \{\ba\in\Omega\,|\, f(\ba) = 0\}\subset\Omega\).
A smooth function \(f\colon\Omega\to\RR\) is said to have simple zeros if \(\nabla f|_\ba \neq 0\) for all \(\ba\in\Sigma_{f}\).  Geometrically, by the implicit function theorem, the zero set \(\Sigma_f\) of a function with simple zeros forms a smooth hypersurface.
\begin{theorem}[Smooth quotient of real functions in arbitrary dimension]
\label{thm:RealQuotientnD}
    Let \(f,g\in C^\infty(\Omega)\) be smooth real-valued functions with simple zeros and a common zero set \(\Sigma_{f} = \Sigma_{g}\).  Then there exists a non-vanishing smooth function \(\varphi\in C^\infty(\Omega)\) such that \(f = g\varphi\).
\end{theorem}
A proof of \autoref{thm:RealQuotientnD} is provided in \autoref{app:ProofOfRealQuotientnD}.

One might ask whether a similar result holds for complex-valued  functions in arbitrary dimension. Specifically, do two complex-valued functions on an open set \(\Omega\subset\RR^n\) with a common zero set always have a smooth complex-valued quotient? The answer to this question is surprisingly different from the real-valued case.

As we will see in Section \ref{sec:RatioOfComplexFunctions}, having common simple zeros is generally insufficient. It turns out that the existence of a continuous quotient requires an additional algebraic relation between the gradients of the functions at the zero set. We describe this condition in the following.



\section{Quotient of Complex Functions}
\label{sec:RatioOfComplexFunctions}

Most existing generalizations of the l'H\^opital theorem for complex-valued functions are formulated on a real interval \cite{Carter:1958:HRC} or for holomorphic functions on a complex plane, a restrictive class of functions studied in complex analysis \cite{AhlforsLars:1979:complex}.
In the latter case, the quotient $f/g$ is globally defined and it coincides with $f'(\ba)/g'(\ba)$ at each common simple zero $\ba$. 

Of course, a theory that guarantees the existence of a quotient should be much more general than simply requiring both functions to be locally holomorphic. For instance, the quotient \(f/f = 1\) is smooth even when \(f\) is not holomorphic. 
Moreover, holomorphy is not invariant under smooth coordinate transformations, whereas the existence of a smooth quotient is independent. Additionally, it is unclear how the notion of holomorphy should make sense in a domain of general real dimension. 

So, can we generalize the l'H\^opital theorem for complex-valued functions in a space of arbitrary dimension?  This is the main question we explore in this article. 


Let \(\Omega\subset\RR^n\) be an open set with \(n\geq 2\). For a complex-valued function \(f\colon\Omega\to\CC\), the zero set is defined as \(\Gamma_f \coloneqq \{\ba\in\Omega\,|\, f(\ba) = 0\}\subset\Omega\), 
and the notion of simple zeros for complex-valued functions is characterized by a non-degeneracy condition on the first derivative.

\begin{definition}[Simple zero]
For $n\geq 2$, a differentiable function \(f\colon\Omega(\subset \RR^n) \to\CC\) is said to have \emph{simple zeros} if, at each zero \(\ba\in\Gamma_f\), the derivative \(Df|_{\ba}\colon\RR^n\xrightarrow{\rm linear}\CC\cong\RR^2\), considered as a real linear map, has rank 2. 
\end{definition}

By the implicit function theorem, the zero set \(\Gamma_f\) of a function with simple zeros forms a codimension-2 submanifold embedded in \(\Omega\).

\begin{problem}[Smooth quotient of complex multivariable functions]\label{prob:ComplexLHopital}
Let \(f,g\in C^\infty(\Omega;\CC)\) be smooth complex-valued functions with a common zero set \(\Gamma_f = \Gamma_g\) consisting of simple zeros.  Does there exist a nonvanishing complex-valued function \(\varphi\in C^\infty(\Omega;\CC)\) so that \(f = g\varphi\)?
\end{problem}
In most cases, the answer to \Cref{prob:ComplexLHopital} is negative, as demonstrated in the following examples.
\begin{example}\label{eg:NonexampleRatio1}
    Let \(\Omega=\RR^2\), \(f(x,y)\coloneqq x+y+\ii y\), and \(g(x,y) = x + \ii y\).  They have a common simple zero at the origin.  The quotient of these two functions is given by
    \(
    \varphi(x,y) = \frac{f(x,y)}{g(x,y)} = \frac{x+y+\ii y}{x+\ii y}. 
    \)
    This function is constant along any straight line passing through the origin, yet is undefined at the origin. To see this, let us take a line $(x(t),y(t))=(t,ct)$ with some real constant $c$ and a parameter $t$. Then we have \(\varphi(t,ct)=\frac{t+ct+\ii ct}{t+\ii ct}=\frac{1+c+\ii c}{1+\ii c}\) and the value at the origin depends on $c$, as shown in \figref{fig:discont_quotient}.
\end{example}
\begin{example}\label{eg:NonexampleRatio2}
    Let \(\Omega = \RR^2\), \(f(x,y) = x-\ii y\), and \(g(x,y) = x+\ii y\), which have a common simple zero at the origin.  Their quotient is given by \(\varphi(x,y) = \frac{f(x,y)}{g(x,y)} = \frac{x-\ii y}{x+\ii y}\). Just as \autoref{eg:NonexampleRatio1},  this function is also constant along any line passing through the origin but is discontinuous at the origin.
\end{example}

In Examples~\ref{eg:NonexampleRatio1} and \ref{eg:NonexampleRatio2}, the quotient \(f/g\) has a well-defined l'H\^opital limit at the origin when restricted to a one-dimensional path approaching the origin (\eg, along the real axis), as studied in \cite{Carter:1958:HRC}.  However, \(f/g\) is not continuous as a function of two variables.

\vspace{-40pt}
\begin{figure}[h]
    \centering
    \begin{minipage}[t]{0.4\textwidth}
    \begin{picture}(210,210)
        \put(0,0)
        {\includegraphics[width=1.0\columnwidth,trim={4cm 0cm 4cm 0cm},clip]{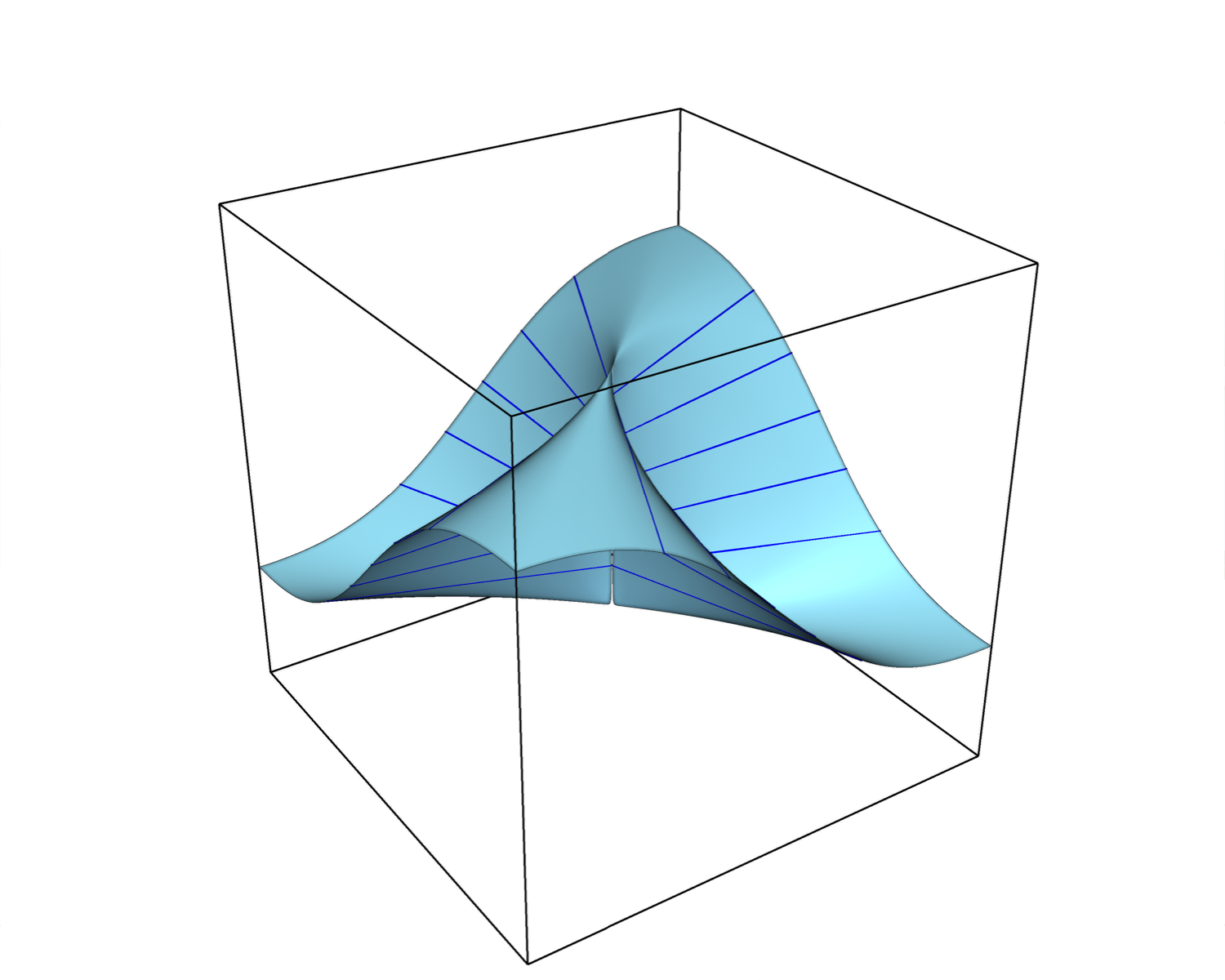}}
        
      \put(95,12){\small\(x\)}
       \put(35,20){\small\(y\)}
    \end{picture}
    \caption{The real part of the quotient of \(f(x,y)=x+y+\ii y\) and $g(x,y)=x+\ii y$. Blue lines are  level lines, showing that $f/g$ takes a constant value along any line passing through the origin.}
    \label{fig:discont_quotient}
     \end{minipage}
         \hspace{20pt}
          \begin{minipage}[t]{0.4\textwidth}
    \begin{picture}(210,210)
        \put(0,0)
        {\includegraphics[width=1.0\columnwidth,trim={4cm 0cm 4cm 0cm},clip]{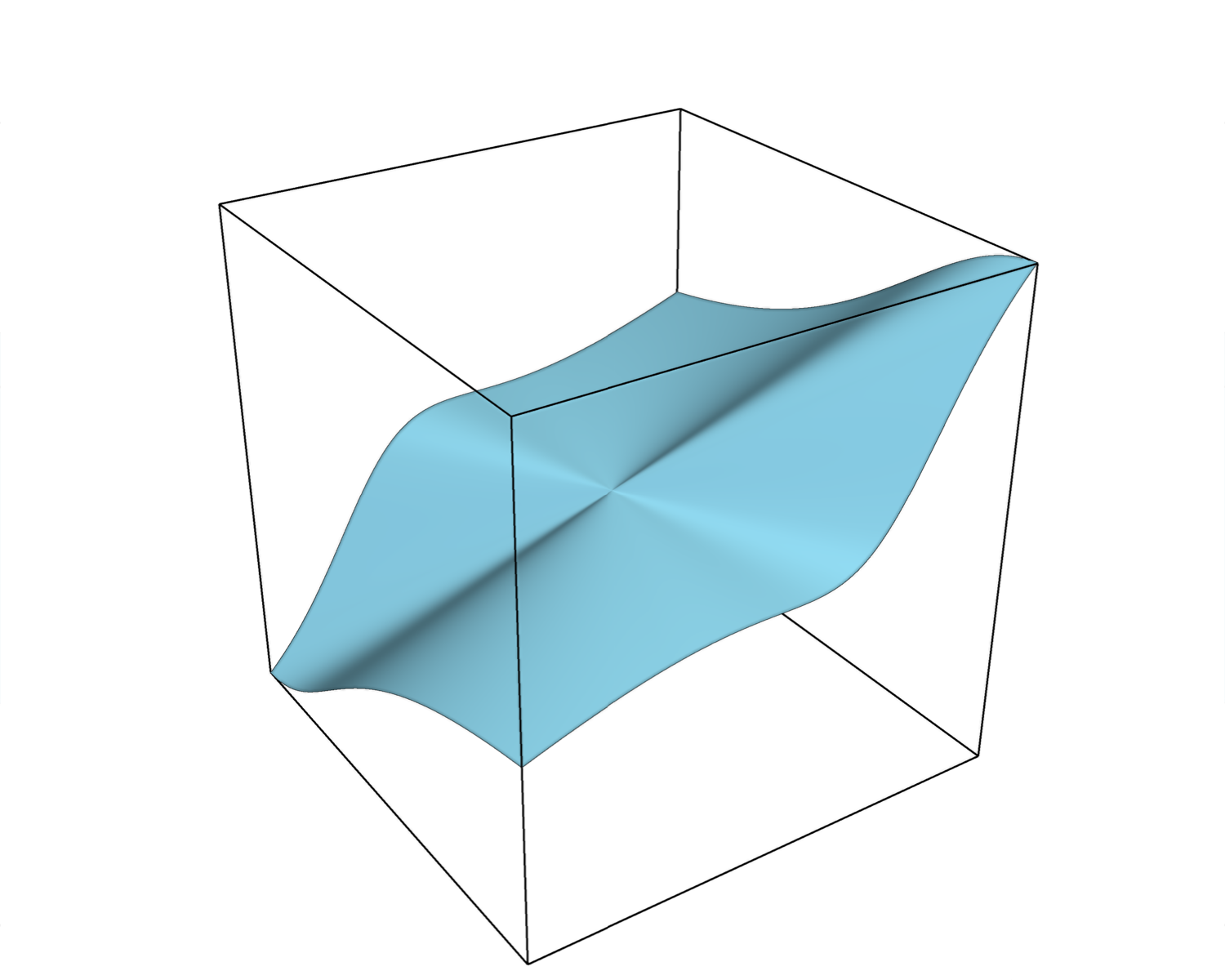}}
     \put(95,12){\small\(x\)}
       \put(35,20){\small\(y\)}
    \end{picture}
    \caption{The real part of the quotient of 
     \(f(x,y)= x+x^2+\ii  x+\ii y\) and \(g(x,y) = x + \ii x+\ii y\). The function $f/g$ is continuous but not differentiable at the origin.}
    \label{fig:c0_quotient}
     \end{minipage}
\end{figure}

\begin{example}\label{eg:c0_quotient}
        Let \(\Omega=\RR^2\), \(f(x,y)= x+x^2+\ii  x+\ii y\), and \(g(x,y) = x + \ii x+\ii y\).
        This pair defines a continuous quotient,
         although it is not differentiable at the origin, as seen in \figref{fig:c0_quotient}. 
\end{example}
This example defines a continuous quotient $f/g$ unlike the previous two examples. We now investigate what makes the difference.

\section{Ratio of derivatives}
Similar to the classical l'H\^opital's rule, the behavior of the quotient of two functions near their common zero is closely related to the quotient of their derivatives. In this section, we explore a more intricate structure that arises near a common simple zero of two complex-valued functions.

Let \(f,g\in C^\infty(\Omega;\CC)\) be smooth complex-valued functions with simple zeros and a common zero set \(\Gamma\).
Then, for all \(\ba\in\Gamma\), the derivatives \(Df|_{\ba}\) and \(Dg|_{\ba}\) are surjective real linear maps to \(\CC\cong\RR^2\).
This defines a unique real matrix \(\bA|_{\ba}\in\RR^{2\times 2}\) (an endormorphism on \(\CC\cong\RR^2\)) that relates the two real linear operators via
\begin{align}
\label{eq:DefinitionOfA}
    Df|_{\ba} = \bA|_{\ba} Dg|_{\ba}.
\end{align}
The matrix representation \(\bA|_{\ba}\) of this relationship depends only on the choice basis for \(\CC\cong\RR^2\), which is taken to be \(1\) and \(\ii\).  In particular, the matrix \(\bA|_{\ba}\) is independent of the choice of coordinate system for \(\Omega\).  

If a change of coordinate \(\Phi\colon\tilde\Omega\to\Omega\) is applied, where \(\Phi(\tilde\ba) = \ba\), and we define \(\tilde f\coloneqq f\circ\Phi\) and \(\tilde g\coloneqq g\circ\Phi\), then the derivatives transform as \(D\tilde f|_{\tilde\ba} = Df|_{\ba} D\Phi|_{\tilde\ba}\) and \(D\tilde g|_{\tilde\ba} = Dg|_{\ba}D\Phi|_{\tilde\ba}\). 
In particular, the relation \(D\tilde f|_{\tilde\ba} = \bA|_{\tilde\ba} D\tilde g|_{\tilde\ba}\) still holds in the transformed coordinates.
That is, the matrix \(\bA\) remains invariant in this change of coordinate. 
\begin{definition}[Ratio of derivatives]
    The function \(\bA\colon\Gamma\to\RR^{2\times 2}\) defined in \eqref{eq:DefinitionOfA} is denoted by
    \begin{align}
    \label{eq:RatioOfDerivatives}
        \left(\frac{Df}{Dg}\right)_{\ba}\coloneqq\bA|_{\ba},\quad \text{for } \ba\in \Gamma.
    \end{align}
\end{definition}

While in general \(f/g\) is discontinuous at \(\Gamma\), as demonstrated in Examples~\ref{eg:NonexampleRatio1} and \ref{eg:NonexampleRatio2}, the nature of the discontinuity is such that the directional limit of \(f/g\) at a point in \(\Gamma\) depends on the path of approach.  This path dependency can be described in terms of the quotient of derivatives \eqref{eq:RatioOfDerivatives}.
\begin{theorem}[Path dependent l'H\^opital rule]\label{thm:PathDependentlHopital}
    Let \(f,g\in C^\infty(\Omega;\CC)\) be smooth complex-valued functions with simple zeros and share a common zero set \(\Gamma\).  
    Suppose \(\gamma\colon (-\varepsilon,\varepsilon)\to\Omega\) is a smooth path such that \(\gamma(0) \in \Gamma\) and \(\gamma'(0)\) is transversal to \(\Gamma\), meaning that \(g_\gamma'\coloneqq \lim_{t\to 0} g(\gamma(t))/t=Dg|_{\gamma(0)}\gamma'(0)\in\CC\) is nonzero.  Then the directional limit of \(f/g\) along \(\gamma\) exists and is given by 
    \begin{align}
    \label{eq:PathDependentlHopital}
        \lim_{t\to 0} \frac{f(\gamma(t))}{g(\gamma(t))} = {1\over g_\gamma'}\left(\frac{Df}{Dg}\right)_{\gamma(0)}g_\gamma'.
    \end{align}
    Here, \(\left(\frac{Df}{Dg}\right)_{\gamma(0)}g_\gamma'\) represents the action of the \(\RR^{2\times 2}\) matrix \({Df\over Dg}\) on \(g_\gamma'\) seen as a real two dimensional vector, while division by  \(g_\gamma'\) is interpreted in terms of complex numbers.  Moreover, the quotient \(f/g\) remains smooth when restricted to the path, meaning that there exists a smooth nowhere-vanishing function \(\varphi\in C^\infty((-\epsilon,\epsilon);\CC)\) such that \(f(\gamma(t)) = g(\gamma(t))\varphi(t)\).
\end{theorem}
We give a proof in \autoref{app:proof_PathDependentlHopital}.
Example \autoref{eg:NonexampleRatio1} and \autoref{eg:NonexampleRatio2} in the previous section fall onto the case of \autoref{thm:PathDependentlHopital} and their path-wise limits at the zero is given as in \eqref{eq:PathDependentlHopital}.

\section{Continuous Quotient of Complex Multivariable Functions}
\label{sec:QuotientOfComplexMultivariableFunctions}
The limit \eqref{eq:PathDependentlHopital} for the quotient \(f/g\) of two complex-valued functions with a common simple zero set \(\Gamma\) is generally path-dependent unless the linear operator \(({Df\over Dg})\) satisfies additional conditions.
\begin{definition}[Complex linearity]
\label{def:ConformalEquivalence}
    Two smooth complex-valued functions \(f,g\in C^\infty(\Omega,\CC)\) with simple zeros and a common zero set \(\Gamma\) are said to be \emph{complex linearly related} if the quotient of their derivatives, \(({Df\over Dg})\colon\Gamma\to\RR^{2\times 2}\), is a scaled rotation at every point in \(\Gamma\).  That is, \(({Df\over Dg})\) takes the form
    \begin{align}\label{eq:ScaledRotation}
        \left({Df\over Dg}\right)_{\ba} = 
        \begin{bmatrix}
            u(\ba)&-v(\ba)\\
            v(\ba)&u(\ba)
        \end{bmatrix},\quad\ba\in\Gamma,
    \end{align}
    for some functions \(u,v\colon\Gamma\to\RR\).
\end{definition}

Matrices of the form \eqref{eq:ScaledRotation} can be rewritten as \(
\begin{bsmallmatrix}
    u&-v\\
    v & u
\end{bsmallmatrix}
= r
\begin{bsmallmatrix}
    \cos\theta & -\sin\theta\\
    \sin\theta & \cos\theta
\end{bsmallmatrix}
\) where \(r = \sqrt{u^2 + v^2}\) and \(\theta = \arctan{v\over u}\), showing that the transformation represents a scaled rotation.
Applying the matrix \(\begin{bsmallmatrix}
    u&-v\\
    v & u
\end{bsmallmatrix}\) to an \(\RR^2\) vector is equivalent to complex multiplication by \(u+\ii v = re^{\ii\theta}\).
Another equivalent characterization of a matrix being a scaled rotation is that its two singular values are equal and its determinant is positive.

\begin{theorem}[Continuous quotient of complex functions in arbitrary dimension]\label{thm:ComplexQuotientnD}
    Let \(f,g\in C^1(\Omega;\CC)\) with simple zeros and a common zero set \(\Gamma\).  Then the following statements are equivalent:
    \begin{enumerate}[leftmargin=*,label=(\roman*)]
        \item 
        There exists a non-vanishing continuous function \(\varphi\in C^0(\Omega)\) such that \(f = g\varphi\).       
        \item \(f\) and \(g\) are complex linearly related.
    \end{enumerate}
\end{theorem}
We give a proof in \autoref{app:proof_ ComplexQuotientnD}. The following corollary is an immediate consequence of \autoref{thm:ComplexQuotientnD}.

\begin{corollary}\label{cor:ComplexLinearEquivalenceRelation}
    The notion of complex linearity relation, as defined in \Cref{def:ConformalEquivalence}, constitutes an equivalence relation for functions sharing common simple zeros.
\end{corollary}
\begin{remark}
    Higher regularity of $f$ and $g$ does not, in general, improve the regularity of the quotient in \autoref{thm:ComplexQuotientnD}. Let us consider the pair $f(x,y)=x+x^2+\ii x+\ii y$ and $g(x,y)=x+\ii x + \ii y$ given in Example \ref{eg:c0_quotient}. Its quotient of derivative $Df/Dg$ is an identity matrix at $(0,0)$, hence \autoref{thm:ComplexQuotientnD} applies. But the quotient is merely $C^0$ although both $f$ and $g$ are smooth, as illustrated in \figref{fig:c0_quotient}.
\end{remark}

\subsection{Open problem}
At this point we are unaware of the exact condition for the existence of a smooth quotient. An obvious sufficient condition is that the complex linearity of $Df/Dg$ (\ie, it defines a complex number) extends to an open neighborhood of $\Gamma$. In 2D, this amounts to defining $f$ by multiplying a nowhere vanishing holomorphic function with $g$.

This, of course, is not a necessary condition. For example, consider the pair $f(x,y)=(x+\ii y)e^{\ii(x^2+y^2)},g(x,y)=x+\ii y$, and the pair $f(x,y)=(x+\ii y)(1+x-\ii y)$ (near the origin) and $g(x,y)=x+\ii y$. They both have the  derivative ratios $Df/Dg$ that are complex linear  but do not define a complex number anywhere around the origin. Nevertheless, they both yield smooth quotients. 
We invite the reader to take up the challenge of finding a sharper condition for the existence of a smooth quotient.

\section{Generalizations}
\label{sec:Generalizations}
The results of Theorems~\ref{thm:RealQuotientnD}, \ref{thm:PathDependentlHopital} and \ref{thm:ComplexQuotientnD} remain valid under any smooth change of coordinates in \(\Omega\in\RR^n\) and do not depend on the Euclidean structure of \(\RR^n\).  Consequently, the domain \(\Omega\) in these theorems can be replaced by any smooth manifold.

Additionally, the assumption of \(C^\infty\)-smoothness for the given functions \(f,g\) in all statements throughout this article can be relaxed to \(C^k\) regularity for any \(k\geq 1\), with the corresponding regularity of \(\varphi\) adjusted to \(C^{k-1}\) accordingly. In fact, the proofs we give in \autoref{app:FirstOrderTaylorExpansion} to \ref{app:proof_ ComplexQuotientnD} are all valid for the general  \(C^k\) regularity case.

\ifdefined\AMM
\begin{acknowledgment}{Acknowledgments.}
This project was funded in part by the European Research Council (ERC Consolidator
Grant 101045083 CoDiNA) and the National Science Foundation CAREER Award 2239062.
\end{acknowledgment}
\else
\begin{acknowledgment}{Acknowledgments.}
This project was funded in part by the European Research Council (ERC Consolidator
Grant 101045083 CoDiNA) and the National Science Foundation CAREER Award 2239062.
\end{acknowledgment}
\fi

\bibliographystyle{vancouver}
\bibliography{Reference}
\appendix

\section{First Order Taylor Expansion}
\label{app:FirstOrderTaylorExpansion}
The smoothness statements in Theorems~\ref{thm:RealQuotient1D}--\ref{thm:PathDependentlHopital} follow from the regularity of the remainder terms in the Taylor expansions.

Suppose \(f\colon I\to\RR\) is a continuously differentiable function of one real variable over an open interval \(I\ni a\).  Then \(f\) can be expressed using the first-order Taylor formula as \(f(x) = f(a) + f'(a)(x-a) + (x-a)h(x)\), where the function \(h\) satisfies \(\lim_{x\to a}h(x) = 0\).  This form of the Taylor remainder is known as \emph{Peano's form}, first characterized by 
Giuseppe Peano in 1889 \cite{Peano:1889:UNF,Persson:2017:HST}.   

If we further assume that \(f\in C^\infty(I)\), then the Peano remainder function \(h\) is also \(C^\infty\)-smooth.  This regularity result is known as Hadamard's Lemma, which states that \(f(x) = f(a) + (x-a)g(x)\) for some \(g\in C^\infty(I)\).  An elementary proof of Hadamard's lemma can be found in \cite{Nestruev:2003:SMO}.

More generally, if \(f\in C^k(I)\) for some \(k\geq 1\), then the first-order Peano remainder function \(h(x)\) (or \(g(x) = h(x)+f'(a)\) in the form of Hadamard's lemma) belongs to \(C^{k-1}(I)\).  This was shown by Hassler Whitney in 1943 \cite{Whitney:1943:DRT}.

\begin{lemma}[Whitney 1943 \cite{Whitney:1943:DRT}]
\label{lemma:Whitney}
    Suppose \(f\in C^k(I)\) with \(k\geq 1\) on an open interval \(I\ni 0\).  Then there exists a function \(g\in C^{k-1}(I)\) such that
    \(
        f(x) = f(0) + xg(x),
    \)     
    with \(g^{(j)}(0) = {1\over j+1}f^{(j+1)}(0)\) for \(0\leq j\leq k-1\).
\end{lemma}
\begin{proof}
    For \(x\neq 0\), the finite difference quotient \(g = {f(x) - f(0)\over x}\) can be written as \(g(x) = {1\over x}\int_0^xf'(t)\, dt\).   The derivatives of \(g(x)\) admit the following integral formula:
    \begin{align*}
        {d^j\over dx^j}g(x) = {1\over x^{j+1}}\int_0^x t^j f^{(j+1)}(t)\, dt,\quad j=0,\ldots,k-1.
    \end{align*}
    This identity can be verified by repeated integration by parts and mathematical induction; we omit the details here.  
    Since \(f\in C^k(I)\), it follows that \(g\in C^k(I\setminus\{0\})\subset C^{k-1}(I\setminus\{0\})\).
    It remains to show that \(g^{(j)}(x)\) has a limit as \(x\to 0\), for \(0\leq j\leq k-1\).
    Define \(L_j \coloneqq {1\over j+1}f^{(j+1)}(0)\), which can also be expressed as \(L_j = {1\over x^{j+1}}\int_0^x t^jf^{(j+1)}(0)\, dt \) for any \(x\neq 0\).  Then we compute:
    \begin{align*}
        \left|g^{(j)}(x) - L\right| & = 
        \left|{1\over x^{j+1}}\int_0^x t^j\left(f^{(j+1)}(t) - f^{(j+1)}(0)\right)\, dt\right|\\
        &\leq \underbrace{{1\over x^{j+1}}\int_0^x t^j\, dt}_{{1\over j+1}} \max_{t\in [0,x]}\left|f^{(j+1)}(t) - f^{(j+1)}(0)\right|.
    \end{align*}
    Since \(f^{(j+1)}\) is continuous at \(0\), this expression tends to zero as \(x\to 0\).  Therefore, each \(g^{(j)}\) extends continuously to \(x=0\), with \(g^{(j)}(0) = L_j\).  Hence, \(g\in C^{k-1}(I)\).
\end{proof}

Whitney also extended the result of \Cref{lemma:Whitney} to multidimensions in \cite{Whitney:1943:DRT}.  Suppose \(f\colon\RR^n\to\RR\) is defined in a neighborhood of the origin.  Then by fixing each \(x_2,\ldots,x_n\) and applying the same argument as \Cref{lemma:Whitney} on the variable \(x_1\), we obtain
\begin{align}
\label{eq:WhitneyTowardsMultidiemsion}
        f(x_1,x_2,\ldots,x_n) = f(0,x_2,\ldots,x_n) + x_1 g_1(x_1,x_2,\ldots,x_n)
\end{align}
where \(f(0,x_2,\ldots,x_n)\) is still \(C^k\) and \(g_1\in C^{k-1}\).  The continuity of \({\partial^{i_1+\cdots+i_n}\over\partial x_1^{i_1}\cdots \partial x_n^{i_n}}g_1\) for \(i_1+\cdots+i_n\leq{k-1}\) is the result of applying \Cref{lemma:Whitney} on the function \({\partial^{i_2+\cdots +i_n}\over\partial x_2^{i_2}\cdots\partial x_n^{i_n}}f\).

Repeating the process of \eqref{eq:WhitneyTowardsMultidiemsion} yields
\begin{align}
\label{eq:WhitneyMultidimensionScalar}
        f(\bx) = f(\bzero) + x_1g_1(x_1,\ldots,x_n) + x_2g_2(x_2,\ldots,x_n)+\cdots+ x_n g_n(x_n)
\end{align}
where \(g_i\in C^{k-1}\) and \(g_i(\bzero) = {\partial f\over\partial x_i}(\bzero)\).

The result \eqref{eq:WhitneyMultidimensionScalar} also holds for vector-valued functions \(\bff = \begin{bsmallmatrix}
    f_1\\|\\f_m
\end{bsmallmatrix}\colon\RR^n\to\RR^m\) by applying \eqref{eq:WhitneyMultidimensionScalar} component-wise, which is concluded by the following lemma.
In the lemma we single out the factors \(\sum_{i}{\partial\bff\over\partial x_i}(\bzero)x_i = D\bff(\bzero)\bx\) to resemble Peano's form of Taylor expansion.

\begin{lemma}[First-order Taylor formula with regularity]
\label{lemma:Taylor}
Let \(\bff\colon U\subset\RR^n\to\RR^m\) be a \(C^k\)-continuous vector-valued function on a neighborhood \(U\) of a given point $\ba\in \RR^n$.
Then there exists \(\bH\in C^{k-1}(U;\RR^{m\times n})\) matrix-valued function so that \(\bH(\ba) = \bzero\) and
\begin{align*}
    \bff(\bx) = \bff(\ba) + D\bff(\ba)(\bx-\ba) + \bH(\bx)(\bx-\ba),\quad\bx\in U.
\end{align*}
\end{lemma}

\section{Proof of Theorem \ref{thm:RealQuotient1D}}
\label{app:ProofOfRealQuotient1D}

\begin{proof}[Proof of \autoref{thm:RealQuotient1D}]
    Let \(\Gamma\) be the zero set of \(f\in C^k(I)\) and \(g\in C^k(I)\).  For \(x\in I\setminus\Gamma\) we define \(\varphi\) by \(\varphi(x) = f(x)/g(x)\), which belongs to \(C^{k}(I\setminus\Gamma)\subset C^{k-1}(I\setminus\Gamma)\) and non-vanishing.  It remains to show that \(\varphi\) is \(C^{k-1}\)-smoothly defined in a neighborhood of \(\Gamma\).  
    By \Cref{lemma:Whitney}, near a simple zero \(a \in \Gamma\) of both \(f\) and \(g\), we can express  
\[
f(x) = f'(a)(x-a)F(x), \quad g(x) = g'(a)(x-a)G(x)
\]
for \(x\) in a neighborhood \(U\) of \(a\), where \(F,G\in C^{k-1}(U)\), \(F(a) = G(a) = 1\), and both \(F(x)\) and \(G(x)\) are nonzero on \(U\). 
Thus, for \(x \in U\), we define  
\[
\varphi(x) \coloneqq \frac{f'(a)}{g'(a)} \cdot \frac{F(x)}{G(x)},
\]
which is smooth, non-vanishing,  $\varphi(a)=f'(a)/g'(a)$, and agrees with \(f(x)/g(x)\) on \(U\setminus\{a\}\).  
\end{proof}

\section{Proof of Theorem \ref{thm:RealQuotientnD}}
\label{app:ProofOfRealQuotientnD}
\begin{proof}[Proof of \autoref{thm:RealQuotientnD}]
    Let \(\Sigma\) be the common zero set of \(f,g\in C^k(\Omega)\).  On \(\Omega\setminus\Sigma\) define \(\bx\in\Omega\setminus\Sigma\) define \(\varphi(\bx) = f(\bx)/g(\bx)\), which belongs to \(C^{k}(\Omega\setminus\Sigma)\) and is non-vanishing.
    It remains to show that \(\varphi\) is \(C^{k-1}\)-continuously defined in a neighborhood of any point of \(\Sigma\).  Let \(\ba\in\Sigma\).  By the implicit function theorem, there exists a neighborhood \(U\) of \(\ba\), a neighborhood \(\tilde U\subset\RR^n\) of the origin \(\bzero\in\RR^n\), and a \(C^k\) bijection \(\Phi\colon \tilde U\xrightarrow{\simeq} U\) such that \(\Phi(\bzero) = \ba\) and \(\tilde\Sigma = \Phi^{-1}(\Sigma)\) is the coordinate hyperplane \(\tilde\Sigma = \{(x_1,\ldots, x_n)\,|\,x_1 = 0\}\cap\tilde U\).
    Define \(\tilde f = f\circ\Phi\) and \(\tilde g = g\circ\Phi\), both of which are \(C^k\) functions on \(\tilde U\) and have simple zeros on \(\tilde\Sigma\).  
    In this coordinate system, \(\nabla \tilde f(0,x_2,\ldots,x_n) = {\partial \tilde f\over\partial x_1}(0,x_2,\ldots,x_n)\be_1\) and \(\nabla \tilde g(0,x_2,\ldots,x_n) = {\partial \tilde g\over\partial x_1}(0,x_2,\ldots,x_n)\be_1\), where \(\be_1\) is the coordinate vector \((1,0,\ldots,0)\), and \({\partial\tilde f\over\partial x_1}\) and \({\partial\tilde g\over\partial x_1}\) are nonvanishing \(C^{k-1}\) function on \(\tilde\Sigma\).
    By \eqref{eq:WhitneyTowardsMultidiemsion}, there exists \(p,q\in C^{k-1}(\tilde U)\) such that for \((x_1,\ldots,x_n)\in \tilde U\)
    \begin{align*}
        \tilde f(x_1,\ldots,x_n) = x_1 p(x_1,\ldots,x_n),\quad\tilde g(x_1,\ldots,x_n) = x_1q(x_1,\ldots,x_n)
    \end{align*}
    and \(p = {\partial\tilde f\over\partial x_1}\neq 0\) and \(q = {\partial\tilde g\over\partial x_1}\neq 0\) on \(\tilde\Sigma\).
    Therefore there exists a neighborhood of \(\tilde\Sigma\), without loss of generality \(\tilde U\), on which \(p\) and \(q\) are nonvanishing.  Therefore, \({\tilde f\over\tilde g} = {p\over q}\) is a nonvanishing \(C^{k-1}\) continuous function on \(\tilde U\) and agrees with \(\varphi\circ \Phi\) on \(\tilde U\).  Thus \(\varphi\) extends to a \(C^{k-1}\)-continuous function in a neighborhood of \(\ba\).
\end{proof}

\section{Proof of Theorem \ref{thm:PathDependentlHopital}}\label{app:proof_PathDependentlHopital}

To prove \Cref{thm:PathDependentlHopital} we generalize \Cref{lemma:Whitney} and \Cref{thm:RealQuotient1D} that apply to complex-valued functions.


\begin{lemma}
\label{lemma:ComplexWhitney}
    Suppose \(f\in C^k(I,\CC)\) with \(k\geq 1\) on an open interval \(I\ni 0\).  Then there exists a function \(g\in C^{k-1}(I,\CC)\)
    \(
        f(x) = f(0) + xg(x),
    \)     
    with \(g^{(j)}(0) = {1\over j+1}f^{(j+1)}(0)\) for \(0\leq j\leq k-1\).
\end{lemma}
\begin{proof}
We apply \Cref{lemma:Whitney} to the real and imaginary parts of \(f\) to obtain the real and imaginary parts of \(g\) holding the stated properties.
\end{proof}

\begin{lemma}[Smooth quotient of complex functions on a real line]\label{lem:ComplexQuotient1D}
     Let \(f,g\in C^\infty(I,\CC)\) be smooth complex-valued functions on a real interval \(I\). Suppose \(f,g\) share common zeros, and their derivatives $f'$ and $g'$ do not vanish on the zero set. Then there exists a non-vanishing smooth complex-valued function \(\varphi\in C^\infty(I,\CC)\) such that \(f = g\varphi\).
\end{lemma}
\begin{proof}
    The proof follows the same steps as the proof of \Cref{thm:RealQuotient1D} (\Cref{app:ProofOfRealQuotient1D}), with \Cref{lemma:Whitney} replaced by \Cref{lemma:ComplexWhitney}.
\end{proof}

\begin{proof}[Proof of Theorem \ref{thm:PathDependentlHopital}]
    Consider smooth functions $f_\gamma(t)\coloneqq f(\gamma(t))$ and $f_\gamma(t)=g(\gamma(t))$ on the interval $(-\varepsilon,\varepsilon)$. 
    By \Cref{lem:ComplexQuotient1D},
    the function $\varphi\coloneqq  f_\gamma/ g_\gamma$ is smooth and nowhere vanishing, with $\varphi(0)=f'_\gamma(0)/g'_\gamma(0)$.
    Noting that $g_\gamma'(0)=Dg|_{\gamma(0)}\gamma'(0)$ and $f_\gamma'(0)=Df|_{\gamma(0)}\gamma'(0)= \frac{Df}{Dg} g_\gamma'(0)$, we obtain the stated expression.
\end{proof}

\section{Proof of Theorem \ref{thm:ComplexQuotientnD}}\label{app:proof_ ComplexQuotientnD}

On a 2D domain, \autoref{thm:ComplexQuotientnD} is a direct consequence of \autoref{thm:PathDependentlHopital}. 
If \(f\) and \(g\) are complex linearly related, 
with  the quotient of derivative
\(({Df\over Dg}) = \begin{bsmallmatrix}
    u&-v\\
    v & u
\end{bsmallmatrix}\) with some $u,v\in \RR$ at each zero, 
then \eqref{eq:PathDependentlHopital} simplifies to 
\begin{align*}
    \lim_{t=0}{f(\gamma(t))\over g(\gamma(t))} = {1\over g'}(u+\ii v)g' = u+\ii v
\end{align*}
which is a single complex number that depends only on the point \(\gamma(0)\) and is independent of the path direction \(\gamma'\). This gives a proof of \autoref{thm:ComplexQuotientnD} in 2D.

In higher dimensions, however, \autoref{thm:ComplexQuotientnD} does not automatically follow from \Cref{thm:PathDependentlHopital}.
Let us consider any 2D plane which transversely intersects $\Gamma$, on which the situation reduces to the 2D setting discussed above. Hence the quotient function $\varphi$ is continuous along any path that intersects $\Gamma$ transversely. In addition, $\varphi$ is smooth when restricted either to $\varphi|_{\Gamma}$ or to $\varphi_{\Omega\setminus\Gamma}$ by construction. These properties, however, are not sufficient to conclude the continuity of $\varphi$. In fact, we can construct a function that satisfies all these properties that is still undefined on $\Gamma$, as the following example illustrates. 
\newline
%
\begingroup
\setlength{\intextsep}{10pt}%
\setlength{\columnsep}{10pt}%
\begin{wrapfigure}{r}{0.2\textwidth}
      \begin{picture}(0,150)
        \put(0,0){\includegraphics[width=.2\columnwidth,trim={13cm 1cm 13cm 0}, clip]{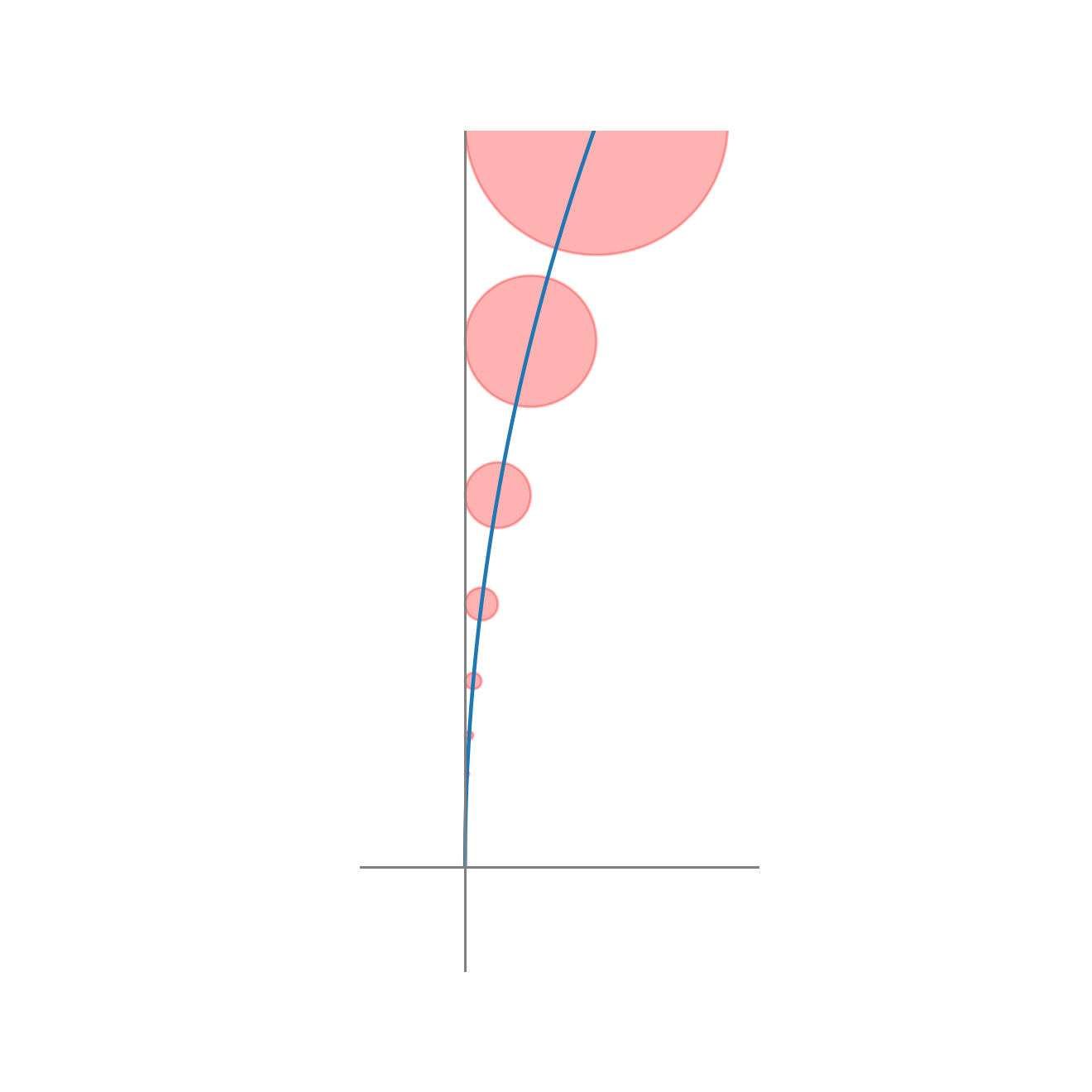}}
        \put(45,22){\small\(x\)}
        \put(16,80){\small\(z\)}
    \end{picture}
\end{wrapfigure}
\begin{example}\label{eg:infinite_bump_functions}
    Consider a collection of smooth bump functions $\{b_n\}_{n\in \NN}$ on $\RR^3$ such that each $b_n$ takes its maximum value $1$ at the center $(2^{-n},0, 2^{1-{n\over 2}})$ of its support, given as an open ball of radius $2^{-n}$. Their centers lie on the curve $z=2\sqrt{x}$, and the supports touch $z$-axis, as visualized in the right image.
    Using these bump functions, let us define $\varphi(x,y,z)= \sum_{n=0}^\infty b_n(x,y,z)$. 
    Notice that any path intersecting $z$-axis transversely passes through only finitely many of the above balls. Hence, along such paths $\gamma:(-\epsilon,\epsilon)\to \RR^3$ with $\gamma(0) = \mathbf{0}$, we have $\lim_{t\to0} \varphi(\gamma(t))=0$. However, any curve $(t,0,a \sqrt{t})$ with $a>2$, intersects these balls infinitely many times as $t\to 0$, making $\varphi$ discontinuous at the origin.
\end{example}
 \endgroup
 We now prove \autoref{thm:ComplexQuotientnD} in arbitrary dimension. The key component is the simple zero condition of $f$ and $g$. It ensures that the leading order of these functions on $\Omega\setminus \Gamma$ consists of the coordinates perpendicular to $\Gamma$. This results in the continuity of the quotient along any path approaching to $\Gamma$, possibly non-transversely, as in Example \ref{eg:infinite_bump_functions}.

\begin{proof}[Proof of \autoref{thm:ComplexQuotientnD}]
{$(i)\to (ii)$.} 
It follows from the $C^1$ regularity of $g$ and the continuity of $\varphi$ that $Df(\ba)=\varphi(\ba)Dg(\ba)$ at each zero $\ba$. This shows that \(Df/Dg = \begin{bsmallmatrix}
    \Re\varphi(\ba)&-\Im\varphi(\ba)\\\Im\varphi(\ba)&\Re\varphi(\ba)
\end{bsmallmatrix}\) and therefore \(f\) and \(g\) are complex linearly related.


\vspace{5pt}
\noindent $(ii)\to (i)$.
    Let \(\Gamma\) be the common zero set of \(f,g\in C^1(\Omega;\CC)\).  
    The quotient \(\varphi(\bx) = f(\bx)/g(\bx)\) for \(\bx\in \Omega\setminus\Gamma\), is \(C^1\)-continuous and nonvanishing.
    It remains to show that \(f/g\) extends to a \(C^{0}\)-function in a neighborhood of any point \(\ba\in \Gamma\).  


    As in the proof of \autoref{thm:RealQuotientnD}, the implicit function theorem assures the existence of a neighborhood \(U\) of \(\ba\), a neighborhood \(\tilde U\) of the origin \(\bzero\in\RR^n\), and a \(C^1\) bijection \(\Phi\colon\tilde U\xrightarrow{\simeq} U\) such that \(\Phi(\bzero) = \ba\) and \(\tilde\Gamma = \Phi^{-1}(\Gamma)\) is the subspace \(\tilde\Gamma = \{\bx=(x_1,x_2,\ldots,x_d)\,|\,  x_1=x_2=0\}\cap\tilde U\). Let us write $\bx^\perp=(x_1,x_2)$ and $\bx^\parallel=(x_3,\ldots,x_n)$ for simplicity.
    Now let \(\tilde f = f\circ\Phi\) and \(\tilde g = g\circ\Phi\), which are \(C^1\) functions vanishing on \(\tilde\Gamma\).

\begin{claim}\label{cl:P and Q}
There exists \(\bP,\bQ\in C^{0}(\tilde U;\RR^{2\times 2})\) such that 
    \begin{align}\label{eq:P and Q}
        \tilde f(\bx) = \bP(\bx)\bx^\perp,\quad
        \tilde g(\bx) = 
        \bQ(\bx)\bx^\perp,
    \end{align}
    with $\bP(\bzero,\bx^\parallel)=\nabla_{\bx^\perp}\tilde f(\bzero,\bx^\parallel)$ and $\bQ(\bzero,\bx^\parallel)=\nabla_{\bx^\perp}\tilde g(\bzero,\bx^\parallel)$ being invertible, where the equalities  in \eqref{eq:P and Q} are the natural identification between $\CC\cong \RR^2$. 
\end{claim}
We will prove this claim at the end.
Then from the continuity of $\bP,\bQ$ and the assumption that \(f\) and \(g\) are complex linearly related, it follows that there exists a continuous complex-valued function \( \lambda\) on $\tilde \Gamma$ such that
     \(\bP = \Lambda\bQ\) on \(\tilde\Gamma\) with the matrix-valued function  $\Lambda=\begin{bsmallmatrix}
         \Re\lambda & -\Im\lambda\\\Im\lambda&\Re\lambda
     \end{bsmallmatrix}$.
    %
Using this, we define a function $\tilde \varphi$ on $\tilde \Gamma$ by
    \begin{align*}
    \tilde\varphi(\bx)=
    \begin{cases}
			\tilde f(\bx)/\tilde g(\bx), & \bx \notin \tilde\Gamma,\\
            \lambda(\bx^\parallel), &  \bx \in \tilde \Gamma
		 \end{cases}
\end{align*}
and show that it is continuous at $\bzero\in \tilde \Gamma$.

Since $\bQ$ is invertible on $\tilde \Gamma$, namely on a small region around $\bzero$, there is a strictly positive constant $c$ such that $|\bQ(\bx)\bx^\perp|\geq c |\bx^\perp|$ in some open neighborhood $U'$ with $\overline{U'}\subsetneq \tilde U$.
Then, we have for $\bx\in U'\setminus\tilde \Gamma$, 
    \begin{align*}
        \left|\tilde\varphi(\bx)-\tilde \varphi(\bzero)\right|
        & =\left|\frac{\tilde f(\bx)}{\tilde g(\bx)}- \lambda(\bzero)\right|\\
        &=\frac{| \bP(\bx)\bx^\perp -\Lambda(\bzero)\bQ(\bx)\bx^\perp|}{|\bQ(\bx)\bx^\perp|}\\
        & \leq  \frac{ \| \bP(\bx) -\Lambda(\bzero)\bQ(\bx)\|\cdot |\bx^\perp|}{c|\bx^\perp|}\\
        & = \frac{\| \bP(\bx) -\Lambda(\bzero)\bQ(\bx)\|}{c} \to 0 \text{ as }\bx\to \bzero.
    \end{align*}
    Here $|\cdot|$ is the absolute value for a complex number and the standard Euclidean norm for an element of $\RR^2$, and  $\|\cdot\|$ is the operator norm defined by $\| A\|=\sup_{|\bz|=1} |A\bz|$ for a linear operator $A\coloneqq \RR^2 \to \RR^2$.  
   With the continuity of $\tilde\varphi|_{\tilde \Gamma}$,
   this shows the continuity of $\tilde\varphi$ at $\bzero$.  Hence \(f/g = \tilde\varphi \circ\Phi^{-1}\) is continuous at \(\ba\).

\vspace{5pt}
\noindent\emph{Proof of \hyperref[cl:P and Q]{Claim}.}  
We construct instances of $\bP,\bQ\in C^0(\tilde U, \RR^{2\times 2})$ with the stated properties.
For each $\bx\in\tilde U$, let us consider a function $F:[0,1]\to \CC\cong \RR^2$ by $F(t)= \tilde f(t\bx^\perp, \bx^\parallel)$. Since $F(0)=\tilde f(\bzero, \bx^\parallel)=0$, we have by the fundamental theorem of calculus that,
\begin{align*}
    \tilde f(\bx)=F(1)=\int^1_0 \frac{dF}{dt} dt 
    =\int^1_0 \nabla_{\bx^\perp}\tilde f(t\bx^\perp,\bx^\parallel) \cdot \bx^\perp dt. 
\end{align*}
 Using this relation, we define $\bP(\bx) = \int^1_0 \nabla_{\bx^\perp}\tilde f(t\bx^\perp,\bx^\parallel) dt$. Clearly $\tilde f(\bx)=\bP(\bx)\bx^\perp$. Note also that $\bP$ is continuous due to the $C^1$ regularity of $\tilde f$, and the invertibility of $\bP(\bzero,\bx^\parallel)=\nabla_{\bx^\perp}\tilde f(\bzero,\bx^\parallel)$ follows from the simple zero assumption of $f$. We can construct $\bQ$ using the same argument.



\end{proof}

\end{document}